\iffalse
\documentclass[letterpaper,10pt,conference]{ieeeconf}
\iffalse
\documentclass[letterpaper,12pt,conference]{ieeeconf}
\onecolumn
\fi
\IEEEoverridecommandlockouts                         
\else
\documentclass[11pt,letterpaper,oneside,reqno]{article}
\fi 

\usepackage{environ}

\usepackage[noadjust]{cite}
\usepackage[colorlinks=true]{hyperref}
\usepackage{amssymb, amsmath, amsfonts}
\usepackage{authblk, dsfont, mathrsfs, color, bm,bbm, url}
\usepackage{graphicx, epstopdf, pgfplots, transparent} 
\usepackage{algorithm, algorithmicx, multirow, titlecaps}
\usepackage{threeparttable}
\usepackage[T1]{fontenc}

\usepackage{subcaption}

\usepackage{tikzsymbols}
\DeclareMathAlphabet{\mathbbold}{U}{bbold}{m}{n}

\usepackage[english]{babel}
\usepackage[noend]{algpseudocode} 

\newcommand*{\minOp}{\operatornamewithlimits{min}\limits}

\newcommand*{\sumOp}{\operatornamewithlimits{\sum}\limits}

\newcommand{\tr}{{\mathsf{T}}}

\newcommand{\zero}{\mathbf{0}}

\iftrue
\newcommand{\eye}{\mathbb{I}}
\else
\newcommand{\eye}{\mathbf{I}}
\fi

\newcommand{\vc}[1]{{ \mathrm{#1} }}
\newcommand{\mx}[1]{{ \mathrm{#1} }}

\newcommand{\norm}[1]{\|#1\|}
\newcommand{\inner}[2]{{ \langle {#1,#2} \rangle}}

\newcommand{\nth}{{\text{\tiny{th}}}}
\newcommand{\trace}{\mathrm{tr}}

\newcommand{\Dscr}{{\mathscr{D}}}

\newcommand{\Lscr}{{\mathscr{L}}}

\newcommand{\Rscr}{{\mathscr{R}}}

\newcommand{\Bcal}{{\mathcal{B}}}
\newcommand{\Ccal}{{\mathcal{C}}}
\newcommand{\Dcal}{{\mathcal{D}}}

\newcommand{\Fcal}{{\mathcal{F}}}

\newcommand{\Hcal}{{\mathcal{H}}}
\newcommand{\Ical}{{\mathcal{I}}}
\newcommand{\Jcal}{{\mathcal{J}}}

\newcommand{\Ncal}{{\mathcal{N}}}

\newcommand{\Ucal}{{\mathcal{U}}}
\newcommand{\Vcal}{{\mathcal{V}}}

\newcommand{\Xcal}{{\mathcal{X}}}
\newcommand{\Ycal}{{\mathcal{Y}}}
\newcommand{\Zcal}{{\mathcal{Z}}}

\newcommand{\Nbb}{{\mathbb{N}}}

\newcommand{\Rbb}{{\mathbb{R}}}
\newcommand{\Sbb}{{\mathbb{S}}}

\newcommand{\bbk}{\mathbbold{k}}

\newcommand{\RR}{\mathbb{R}}

\usepackage{amsthm}

\theoremstyle{plain}
\newtheorem{theorem}{Theorem}

\newtheorem{proposition}[theorem]{Propositon}

\newtheorem{definition}{Definition}

\newtheorem{remark}{Remark}

\newtheorem*{problem*}{Problem}

\let\labelindent\relax
\usepackage{enumitem}
\newlist{todolist}{itemize}{2}
\setlist[todolist]{label=$\square$}
\usepackage{pifont}
\iftrue 
\usepackage{xcolor}
\hypersetup{
	colorlinks,
	linkcolor={blue!55!black},
	citecolor={red!55!black},
	urlcolor={blue!80!black}
}
\usepackage[letterpaper, total={7in, 9in}]{geometry}
\fi
\newcommand{\OmegaROA}{\Omega_{\text{ROA}}}
\newcommand{\nS}{{n_{\text{s}}}}
\newcommand{\nT}{{n_{\text{T}}}}
\newcommand{\nG}{{n_{\text{g}}}}

\iftrue
\newcommand{\Kernel}{\mathbbm{K}}
\newcommand{\kernel}{\mathbbm{k}}
\newcommand{\Hk}{\Hcal_\mathbbm{k}}
\newcommand{\HK}{\Hcal_\mathbbm{K}}
\else
\newcommand{\Kernel}{\mathbb{K}}
\newcommand{\kernel}{\bbk}
\newcommand{\HK}{\Hcal_\mathbb{K}}
\newcommand{\Hk}{\Hcal_{\bbk}}
\fi

\newcommand{\mxW}{\mx{W}}
\newcommand{\mxT}{\mx{T}}
\newcommand{\mxP}{\mx{P}}
\newcommand{\vcx}{\vc{x}}
\newcommand{\vcy}{\vc{y}}
\newcommand{\vcz}{\vc{z}}
\newcommand{\vca}{\vc{a}}
\newcommand{\vcab}{\vc{\mathbf{a}}}
\newcommand{\mxA}{\mx{A}}
\newcommand{\vcb}{\vc{b}}

\newcommand{\mxD}{\mx{D}}
\newcommand{\mxJ}{\mx{J}}

\newcommand{\vck}{\vc{k}}

\newcommand{\mxK}{\mx{K}}
\newcommand{\mxKb}{\mx{\mathbf{K}}}
\newcommand{\mxkb}{\vc{\mathbf{k}}}

\newcommand{\Zbbp}{\mathbb{Z}_{
		\scalebox{0.5}{\text{$\ge 0$}}
}}

\title{\LARGE \bf
Nonlinear System Identification with Prior Knowledge of the Region of Attraction}%

\author{Mohammad~Khosravi\thanks{Corresponding author}~ and~ Roy~S.~Smith%
	\thanks{This research project is part of the Swiss Competence Center for Energy
		Research SCCER FEEB\&D of the Swiss Innovation Agency Innosuisse.}
	\ \thanks{The authors are with Automatic Control Lab,  ETH Zurich, Switzerland\\
	 {\tt\small  \{khosravm,rsmith\}@control.ee.ethz.ch}}
}

\begin{document}
\date{}
\maketitle
\thispagestyle{empty}
\pagestyle{empty}
\begin{abstract}
We consider the problem of nonlinear system identification when prior knowledge is available on the region of attraction (ROA) of an equilibrium point. We propose an identification method in the form of an optimization problem, minimizing the fitting error and guaranteeing the desired stability property. The problem is approached by joint identification the dynamics and a Lyapunov function verifying the stability property. In this setting, the hypothesis set is a reproducing kernel Hilbert space, and with respect to each point of the given subset of the ROA, the Lie derivative inequality of the Lyapunov function imposes a constraint. The problem is a non-convex infinite-dimensional optimization with infinite number of constraints. To obtain a tractable formulation, only  a suitably designed finite subset of the constraints are considered. The resulting problem admits a solution in form of a linear combination of the sections of the  kernel and its derivatives. An equivalent optimization problem with a quadratic cost function subject to linear and bilinear constraints is derived. A suitable change of variable gives a convex reformulation of the problem. To reduce the number of hyperparameters, the optimization problem is adapted to the case of diagonal kernels. The method is demonstrate  by means of an example.	
\end{abstract}

\section{Introduction}\label{sec:int}
The identification of nonlinear systems has received significant attention due to its potential in modeling various phenomena in science and engineering \cite{schoukens2019nonlinear}.
Given the measurement data, the techniques of optimization, statistics, and system identification are to mathematically model the physical systems \cite{LjungBooK2}. In many situations modeling involves more than fitting  nonlinear dynamics to the measurement data; one should include additional features as prior knowledge which are expected according to our understanding of the system.
For example, the system properties like stability, passivity and positivity are already considered for the identification of linear dynamics \cite{pillonetto2014kernel,goethals2003identification,khosravi2019positive}.  

For nonlinear systems, stabilizability of the dynamics is considered as a part of the identification problem in \cite{sattar2020non,singh2019learning}. Identification of a stable nonlinear dynamical system has been studied in \cite{calinon2010learning,khansari2011learning,khansari2014learning,ijspeert2013dynamical,sindhwani2018learning,khansari2017learning} mainly motivated by imitation learning.
In \cite{calinon2010learning}, hidden Markov models and Gaussian mixtures are used for modeling the dynamics.
A similar approach is presented in \cite{khansari2011learning} with guaranteed global stability.
In \cite{khansari2014learning}, a two-stage approach is presented where, first a parametric Lyapunov function as well as a model for the dynamics are learned, and then, the learned dynamics is stabilized using the Lyapunov function.
The approach presented in \cite{ijspeert2013dynamical} models the dynamics as a  weakly nonlinear system which consists of a stable linear part for capturing the baseline behavior, and a nonlinear part to account for more complex phenomena, and  a phase variable for the coupling these two parts. It is shown in \cite{ijspeert2013dynamical} that the derived model is stable and time-varying.
An identification method is introduced for learning a globally stable system in \cite{sindhwani2018learning}. Similar to the current paper, the hypothesis space in \cite{sindhwani2018learning} is a {smooth vector-valued reproducing kernel Hilbert space}  (SVRKHS) \cite{zhou2008derivative,singh2019learning,carmeli2006vector}.
Meanwhile, the stability condition in \cite{sindhwani2018learning} is only imposed locally over the data points by forcing the eigenvalues of the corresponding Jacobian to be negative at sampling points.
 
In this paper, we propose a nonlinear system identification method designed to include the available knowledge on a subset of the region of attraction (ROA) of a stable equilibrium point. 
Assuming that this stability property can be verified by a quadratic Lyapunov function, the problem is then formulated as a joint nonparametric estimation over a hypothesis space for the unknown dynamics, characterized here by a SVRKHS, and also, the space of positive definite matrices in order to determine the Lyapunov function. 
The resulting formulation is a non-convex optimization problem over an infinite dimensional space with infinite number of bilinear constraints, arising from the Lie derivative of the Lyapunov function with respect the points of the given subset of ROA. 
In order to make the problem tractable, we first introduce a suitable finite subset of the given subset of ROA such that verifying the Lie derivative inequality on these points guarantees the desired stability property. Following this, we reformulate the problem into a  finite dimensional optimization problem with a quadratic cost function, and linear and bilinear constraints, modeling the stability of the system at the equilibrium point and in the given region. We prove that this problem admits a solution with a linear parametric representation in terms of the sections of the kernel as well as its derivatives. 
Using a non-obvious change of variables, we derive a convex reformulation of the problem. Following this, in order to mitigate the hyperparameter tuning issue, we present the case for diagonal kernels. The method is demonstrated numerically by means of an example. 

\section{Notations and Preliminaries}
The set of natural numbers, the set of non-negative integers, the set of real numbers, the $n$-dimensional Euclidean space and the space of $n$ by $m$ real matrices are denoted by $\Nbb$, $\Zbbp$,  $\Rbb$, $\Rbb^n$ and $\Rbb^{n\times m}$ respectively.
The identity matrix and zero vector in the Euclidean space are denoted by $\eye$ and $\zero$ respectively.
The set of symmetric positive definite matrices in $\Rbb^{n\times n}$ is denoted by $\Sbb_{++}^n$.
For any pair of symmetric matrices $\mx{X},\mx{Y}\in\RR^{n\times n}$, we write $\mx{X}\succeq \mx{Y}$ if $\mx{X}-\mx{Y}\in\Sbb_{++}^n$.
Given $\mxW\in \Sbb_{++}^n$, $\|\cdot\|_{\mxW}$ is a norm on $\Rbb^n$ defined as $\|\vcx\|_{\mxW}:=\vcx^\tr \mxW \vcx$, for any $\vcx\in\Rbb^n$. When $\mxW=\eye$, we drop subscript $\mxW$. The disk in $\Rbb^n$ with center $\vc{c}$ and radius $r>0$ is denoted by $\Bcal(\vc{c},r)$ and defined as $\Bcal(\vc{c},r):=\{\vc{x}\in\Rbb^n\ | \ \norm{\vc{x}-\vc{c}}< r\}$.
For a vector $\alpha = (\alpha_i)_{i=1}^n\in\Zbbp^n$, we define $|\alpha|:=\sum_{i=1}^n\alpha_i\le s$ and
for a function $f$ of $n$ variables we denote 
the partial derivative $\frac{\partial^{\alpha_1}}{\partial x_1^{{\alpha_1}}}
\ldots\frac{\partial^{\alpha_n}}{\partial x_n^{{\alpha_n}}}
f(\vcx)$ by $\partial_\alpha f(\vcx)$.
Similarly, if $f$ is a function with $k$ multivariable arguments, $\partial^i_\alpha f$ denotes the partial derivative of $f$ with respect to the $i^\nth$ argument.
The derivative operator is denoted by $\mx{D}$, i.e., the derivative of $f$ is shown by $\mx{D}f$. 
The interior of set $\Xcal$ is denoted by $\Xcal^o$.
Let $\Xcal$ be a compact subset of $\Rbb^n$ such that $\Xcal$ is the closure of $\Xcal^o$.
Then, $C^{s}(\Xcal,\Rbb^n)$ is defined as the Banach space of functions $g:\Xcal\to\Rbb^n$ where $\partial_{\alpha} g|_{\Xcal^o}$ is well-defined and has a continuous extension to $\Xcal$, for any $\alpha\in\Zbbp^n$ such that $|\alpha|\le s$. 
The norm on $C^{s}(\Xcal,\Rbb^n)$ is defined as 
$$
\|g\|_{C^{s}(\Xcal,\Rbb^n)}=\sumOp_{\alpha\in\Zbbp^n,|\alpha|\le s}\sup_{\vcx}\|\partial_{\alpha}f(\vcx)\|.
$$
One can define $C^{2s}(\Xcal\times\Xcal,\Rbb^{n\times n})$ similarly.
Let $\Ycal$ be a set and $\Ccal$ be a subset of $\Ycal$. The indicator function of $\Ccal$, denoted by $\Ical_{\Ccal}$, is defined as
$\Ical_{\Ccal}(y) = 0$, if $y\in\Ccal$ and  $\Ical_{\Ccal}(y) = \infty$, otherwise.
\section{Problem Statement}
Let $\Ucal$ be an open domain in $\Rbb^n$ and $f:\Ucal\to \Rbb^n$ be an \underline{\em unknown} vector field defined on $\Ucal$ which is $C^2(\Ucal,\Rbb^n)$. Consider the corresponding dynamical system defined as
\begin{equation}\label{eqn:xdot=f(x)}
\dot{\vc{x}} = f(\vc{x}), \quad \vc{x}(0) = \vc{x}_0,
\end{equation}
where $\vc{x}_0\in\Ucal$ is the initial point. Denote the solution of \eqref{eqn:xdot=f(x)}, at time instant $t\ge 0$, by $\vc{x}(t;\vc{x}_0)$.
Let the origin be an asymptotically stable equilibrium of dynamical system \eqref{eqn:xdot=f(x)}. 
Also, let the corresponding region of attraction (ROA) be denoted by $\OmegaROA$, i.e., we have 
\begin{equation}
\OmegaROA:= \{\vc{x}_0\in\Ucal \ |\ \lim_{t\to \infty}\vc{x}(t;\vc{x}_0)=\zero\}.
\end{equation}
Let $\Omega$ be a \underline{\em known} inner approximation for the region of attraction of the origin. 
More precisely, $\Omega$ is a known compact set with non-empty interior such that $\zero\in\Omega\subseteq\OmegaROA$.

Consider a set of trajectories of system,   
like $\{\vc{x}(\cdot;\vc{x}_0^i) \ |\ 1\le i\le\nT \}$, where the corresponding initial points belong to $\OmegaROA$, i.e.,  $\{\vc{x}_0^1,\ldots,\vc{x}_0^{\nT}\}\subset \OmegaROA$.
For any $i=1,\ldots,\nT$, suppose that the $i^\nth$ trajectory is sampled at time instants 
$0\le t^i_1<t^i_2<\cdots<t^i_{n_i}$ where $n_i\in\Nbb$.
Let $\vc{x}_k^i$ denote $\vc{x}(t_k^i;\vc{x}_0^i)$, for $1\le k\le n_i$.
Given these samples, one can estimate the time derivative of $\vc{x}(\cdot;\vc{x}^i_0)$ at the sampling time instants.
This estimation can be performed using various techniques, e.g. see \cite{wang2019robust} and the references therein, or by simply utilizing a nonlinear regression method and subsequently obtaining the derivatives numerically or analytically. 
Let these estimations be denoted by $\vc{y}_k^i$, for $1\le k\le n_i$.
One should note that $\vc{y}_k^i$ is an approximation of $f(\vc{x}_k^i)$.
Considering these samples of trajectories and their estimated derivatives, we get a set of {\em data}, denoted by $\Dscr$, which contains
$(\vc{x}_k^i,\vc{y}_k^i)$ pairs. For notation simplicity, we drop the superscripts and simply show set $\Dscr$ as $\{(\vc{x}_j,\vc{y}_j)\ |\ 1\le j\le m\}$, where $m=\sum_{i=1}^{\nT}n_i$.

\begin{problem*}
Given that origin is a stable equilibrium point of \eqref{eqn:xdot=f(x)} and 
the set $\Omega$ is provided as the prior knowledge about the region of attraction of the origin, the problem is to estimate the unknown vector field $f$, in a given class of functions $\Fcal\subseteq C^2(\Ucal,\Rbb^n)$, using
the set of data $\Dscr$.
\end{problem*}
In the next section, we introduce a tractable formulation of this problem as a nonparametric estimation.
The formulation can be extended to the case of multiple equilibria and multiple regions of attraction.

\section{Main Results: Identification Method}
We know that $f$ satisfies the constraint that $\vc{x}=\zero$ is an equilibrium point of \eqref{eqn:xdot=f(x)}, 
i.e.,
$f(\zero)=\zero$.
Moreover, we know that $\vc{x}=\zero$ is stable and $\Omega$ is a subset of the corresponding region of attraction. 
Assume that these stability features of $f$ can be verified by an \underline{\em unspecified} quadratic Lyapunov function $V(\vc{x}) = \frac{1}{2}\vc{x}^\tr\mx{P}\vc{x}$ where here $\mx{P}$ is a positive definite matrix.
More precisely, there exist an \underline{\em unknown} $\epsilon>0$ and an \underline{\em unknown} positive definite matrix $\mx{P}\succeq \eye$ such that 
\begin{equation}\label{eqn:xtrPfx_le_eps_normx2}
\mx{D}V(\vc{x}) f(\vc{x}) = \vc{x}^\tr\mx{P}f(\vc{x})\le -\epsilon \|\vc{x}\|^2, \quad \forall \vc{x}\in\Omega.
\end{equation}

In the estimation problem, we need to minimize the fitting error, $\sum_{i=1}^m \|\vc{y}_i-f(\vc{x}_i)\|_{\mxW}^2$, subject to $f\in\Fcal$, $\vc{f}(\zero) =\zero$ and \eqref{eqn:xtrPfx_le_eps_normx2}.
There are two main issues: the correct choice of function class $\Fcal$, and dealing with the (uncountable) infinite number of constraints introduced in \eqref{eqn:xtrPfx_le_eps_normx2}.
These issues are addressed in this section.

\subsection{From Infinite to Finite Number of Constraints}
Since an optimization problem with an infinite-dimensional feasible set and infinite  number of imposed constraints can lead to intractability, particularly when the problem is non-convex as it is here, we need to introduce a finite number of suitable constraints implying \eqref{eqn:xtrPfx_le_eps_normx2}.
To this end, the notion of {\em $(\alpha,\beta)$-grid} is introduced in the next definition.
One should note that, based on the discussed given below, the  $(\alpha,\beta)$-grid is a suitably selected finite subset of $\Omega$ such that verifying stability condition 
on its elements implies the desired stability feature given in \eqref{eqn:xtrPfx_le_eps_normx2}.

\begin{definition}\label{def:alpha_beta_grid}
Let $\{\vc{z}_1,\ldots,\vc{z}_{\nG}\}$ be a finite subset of $\Omega\backslash\{\zero\}$ denoted by $\Zcal$. We say  $\Zcal$ is an {\em $(\alpha,\beta)$-grid} for $\Omega$ if 
\begin{equation}\label{eqn:grid}
\Omega \subseteq 
\bigg{(}\bigcup_{i=1}^{\nG} \Bcal(\vc{z}_i,\alpha\|\vc{z}_i\|) \bigg{)}
\cup 
\Bcal(\zero,\beta).
\end{equation}
\end{definition}
The role of $(\alpha,\beta)$-grid in the estimation problem is shown by the next theorem. 
Define $L_{1,f}$ and $L_{2,f}$ respectively as
\begin{equation}\label{eqn:L1fL2f}
\begin{array}{rll}
L_{1,f} & := & \sup_{\vc{x}\in\Omega}\|\mx{D}f(\vc{x})\|,\\
L_{2,f} & := & \sup_{\vc{x}\in\Omega}
\sup_{\vc{h}_1,\vc{h}_2\in\overline{\Bcal}(\zero,1)}
|\mx{D}^2f(\vc{x})(\vc{h}_1,\vc{h}_2)|.\\
\end{array}
\end{equation}	
Since $f$ is $C^2(\Ucal,\Rbb^n)$, we have that $L_{1,f},L_{2,f}<\infty$.
\begin{theorem}\label{thm:grid_sufficient}
	Let $\Zcal:=\{\vc{z}_1,\ldots,\vc{z}_{\nG}\} \subset \Omega\backslash\{\zero\}$ and $\mx{P}\in\Sbb^n$ be such that 
	\begin{equation}\label{eqn:grid_const}
	\begin{array}{l l}
	\vc{z}_i^\tr\mx{P}f(\vc{z}_i)\le -\|\vc{z}_i\|^2, & \forall i = 1,\ldots,\nG,\\
	\frac{1}{2}(\mx{P}\mx{D}f(\zero)+\mx{D}f(\zero)^\tr \mx{P}) \preceq -\eye,\\
	\mx{P}\succeq \eye,
	\end{array}
	\end{equation}
	Given $\epsilon\in(0,1)$, let $\alpha$ and $\beta$ be real positive scalars where
	\begin{equation}\label{eqn:alpha_beta}
	\alpha \le \bigg{(}\frac{1+L_{1,f}\|\mx{P}\|}{\epsilon+L_{1,f}\|\mx{P}\|}\bigg{)}^{\frac{1}{2}}   -1,
	\quad
	\beta \le \frac{1-\epsilon}{L_{2,f}\|\mx{P}\|},
	\end{equation}
	and $\Zcal$ be an $(\alpha,\beta)$-grid for $\Omega$.
	Then, the following holds
	\begin{equation}\label{eqn:xtrPfx_le_eps_normx2_thm}
	\vc{x}^\tr\mx{P}f(\vc{x})\le -\epsilon \|\vc{x}\|^2, \quad \forall \vc{x}\in\Omega.
	\end{equation}
\end{theorem}
\begin{proof}
Let $\vc{x}\in\Omega$. Since $\Zcal$ is an $(\alpha,\beta)$-grid, then by \eqref{eqn:grid}, we know that $\vc{x}$ either belongs to  $\Bcal(\zero,\beta)$, or it belongs to $\Bcal(\vc{z}_i,\alpha\|\vc{z}_i\|)$, for some $1\le i\le\nG$.\\
\underline{Case I:} Assume that $\vc{x}\in\Bcal(\zero,\beta)$, i.e., $\|\vc{x}\|<\beta$. Since, $f$ is $C^2(\Ucal,\Rbb_{\ge 0}^n)$ and $f(\zero)=\zero$, from Taylor expansion at origin, we have that
\begin{equation}
f(\vc{x}) = \mx{D}f(\zero)\vc{x} + r(\vc{x}),
\end{equation}
where $r:\Ucal\to\Rbb^n$ is a $C^2(\Ucal,\Rbb^n)$ function such that
\begin{equation}\label{eqn:L2f_remainder}
\sup_{\vc{x}\in\Omega, \vc{x}\ne  \zero}\frac{\|r(\vc{x})\|}{{\|\vc{x}\|^2}}\le L_{2,f}.
\end{equation}
Accordingly, one can easily see that
\begin{equation*}
\vc{x}^\tr\mx{P}f(\vc{x}) =
\frac{1}{2}\vc{x}^\tr
\big{(}\mx{P}\mx{D}f(\zero)+\mx{D}f(\zero)^\tr\mx{P}\big{)}\vc{x} + \vc{x}^\tr \mx{P} r(\vc{x}).
\end{equation*}
Due to  \eqref{eqn:grid_const}, \eqref{eqn:L2f_remainder}, and the Cauchy-Schwartz inequality, we have
\begin{equation*}
\vc{x}^\tr\mx{P}f(\vc{x}) 
\le 
-\|\vc{x}\|^2
+ 
\|\mx{P}\|L_{2,f} \|\vc{x}\|^3.
\end{equation*}
According to \eqref{eqn:alpha_beta} and since $\|x\|< \beta$, we have
\begin{equation*}
\vc{x}^\tr\mx{P}f(\vc{x})\le -\epsilon\|\vc{x}\|^2.
\end{equation*}
\underline{Case II:} Assume that  $\vc{x}\in\Bcal(\vc{z}_i,\alpha \|\vc{z}_i\|)$, i.e., $\vc{x} = \vc{z}_i + \vc{a}$ where $\vc{a}$ is a vector such that $\|\vc{a}\|<\alpha \|\vc{z}_i\|$.
From the triangle inequality, we have that
\begin{equation}\label{eqn:_eps_1_alpha2_normzi}
-\epsilon\|\vc{z}_i\|^2 (1+\alpha)^2
\le
-\epsilon(\|\vc{z}_i\| + \|\vc{a}\|)^2
\le -\epsilon\|\vc{x}\|^2.
\end{equation}
Due to \eqref{eqn:L1fL2f}, $L_{1,f}$ is a Lipschitz constant for $f$.
Therefore, we have
\begin{equation}
\|f(\vc{z}_i + \vc{a})-f(\vc{z}_i )\| \le L_{1,f} \|\vc{a}\|,
\label{eqn:L1f_ineq_1}
\end{equation}
and
\begin{equation}
\|f(\vc{z}_i + \vc{a})\|\le 
L_{1,f} \big{(}\|\vc{z}_i\|+\|\vc{a}\|\big{)}, 
\label{eqn:L1f_ineq_2}
\end{equation}
where the second inequality follows from $f(\zero)=\zero$ and the triangle inequality.
From \eqref{eqn:grid_const} and the Cauchy-Schwartz inequality,
we have 
\begin{equation*}
\begin{split}
\vc{x}^\tr\mx{P}f(\vc{x}) 
&=\vc{z}_i^\tr\mx{P}f(\vc{z}_i) 
+\vc{z}_i^\tr\mx{P}(f(\vc{z}_i + \vc{a})-f(\vc{z}_i ))
+\vc{a}^\tr\mx{P}f(\vc{z}_i + \vc{a})
\\ 
&\le -\|\vc{z}_i\|^2 + \|\mx{P} \vc{z}_i\| \|f(\vc{z}_i + \vc{a})-f(\vc{z}_i )\|
+
\|\mx{P}\vc{a}\|\|f(\vc{z}_i + \vc{a})\|.
\end{split}
\end{equation*}
Note that $\|\mx{P}\vc{a}\|\le\|\mx{P}\|\|\vc{a}\|$ and $\|\mx{P}\vc{z}_i\|\le\|\mx{P}\|\|\vc{z}_i\|$, for any $i=1,\ldots,\nG$. 
Therefore, from \eqref{eqn:L1f_ineq_1} and \eqref{eqn:L1f_ineq_2}, 
we have 
\begin{equation*}
\begin{split}
\vc{x}^\tr\mx{P}f(\vc{x}) 
&\le 
-\|\vc{z}_i\|^2 + \|\mx{P}\| \|\vc{z}_i\| \|f(\vc{z}_i + \vc{a})-f(\vc{z}_i )\|
+
\|\mx{P}\|\|\vc{a}\|\|f(\vc{z}_i + \vc{a})\|\\ 
&\le 
-\|\vc{z}_i\|^2 + \|\mx{P}\|\|\vc{z}_i\|\ L_{1,f}\|\vc{a}\|
+
\|\mx{P}\|\|\vc{a}\| \ L_{1,f}(\|\vc{z}_i\|+\|\vc{a}\|).
\end{split}
\end{equation*}
Since $\|\vc{a}\|<\alpha\|\vc{z}_i\|$, one can conclude that
\begin{equation}\label{eqn:xtrPfx_le_normz_i2_blahblah}
\vc{x}^\tr\mx{P}f(\vc{x}) 
< -\|\vc{z}_i\|^2 (1 - 2\|\mx{P}\|L_{1,f}\alpha 
-\|\mx{P}\|L_{1,f}\alpha^2).
\end{equation}
From \eqref{eqn:alpha_beta}, one can see that
\begin{equation}\label{eqn:alpha_eps_ineq}
1 - 2\|\mx{P}\|L_{1,f}\alpha 
-\|\mx{P}\|L_{1,f}\alpha^2 \ge \epsilon(1+\alpha)^2.
\end{equation}
Due to \eqref{eqn:_eps_1_alpha2_normzi}, 
\eqref{eqn:xtrPfx_le_normz_i2_blahblah},
and \eqref{eqn:alpha_eps_ineq}, we have that
\begin{equation}
\vc{x}^\tr\mx{P}f(\vc{x})
<
-\epsilon(1+\alpha)^2\|\vc{z}_i\|^2
\le -\epsilon\|\vc{x}\|^2.
\end{equation}
This concludes the proof.
\end{proof}
\begin{remark}
	Theorem \ref{thm:grid_sufficient} guarantees that in order to satisfy the infinite number of constraints given in \eqref{eqn:xtrPfx_le_eps_normx2}, it is enough to satisfy the finite number of constraints introduced in \eqref{eqn:grid_const}, given a suitable $(\alpha,\beta)$-grid for $\Omega$. 
\end{remark}
\begin{remark}
	One can verify that if \eqref{eqn:xtrPfx_le_eps_normx2_thm} holds for a positive-definite matrix $\mx{P}$, then a scaled version of $\mx{P}$ satisfies the inequalities given in \eqref{eqn:grid_const}. 
	This can be interpreted as the other direction of Theorem \ref{thm:grid_sufficient}.	
\end{remark}
\begin{theorem} \label{thm:grid_existence}
	For any $\alpha,\beta>0$, an $(\alpha,\beta)$-grid exists for $\Omega$.
\end{theorem}
\begin{proof}
Since 
\begin{equation}\label{eqn:Omega_cover}
\Omega\subseteq\cup_{\vc{z}\in\Omega,\vc{z}\ne \zero} \Bcal(\vc{z},\alpha\|\vc{z}\|)\cup\Bcal(\vc{z},\beta)
\end{equation}
and $\Omega$ is a compact set, this open cover  has a finite sub-cover.
Hence, there exist $\{\vc{z}_1,\ldots,\vc{z}_{\nG}\}\subset\Omega\backslash\{\zero\}$ such that \eqref{eqn:grid} holds.	
\end{proof}
\begin{remark}
	For the given $\Omega$, there are infinite choices for $(\alpha,\beta)$-grids.
	Moreover, one can see that taking small values for $\alpha$ and $\beta$ results in fine and large $(\alpha,\beta)$-grid.  	
\end{remark}
\subsection{Identifying the Dynamics in the Smooth Vector-valued Reproducing Kernel Hilbert Spaces}
The function class taken for approximating the unknown vector field is a type of Hilbert spaces called {\em smooth vector-valued reproducing kernel Hilbert spaces} (SVRKHS) which are introduced below (see \cite{carmeli2006vector,zhou2008derivative,singh2019learning} for more details). Based on the suitable structure of SVRKHS, we will prove 
that the problem admits a solution with a specific finite linear parametric form. This allows reducing the optimization problem to the coefficients of this representation and subsequently, a tractable finite-dimensional optimization problem is obtained.

Let $\Xcal$ be a compact subset of $\Rbb^n$ with non-empty interior $\Xcal^o$ such that $\Xcal$ is the closure of $\Xcal^o$  
and  $\Omega\subset\Xcal^o$.
\begin{definition}\label{def:SVRKHS}
A Smooth Vector-valued Reproducing Kernel Hilbert Space (SVRKHS), denoted by $\Hcal$, is a Hilbert space of functions $g\in C^{s}(\Xcal,\Rbb^n)$ such that 
for any $\vcx\in \Xcal$, 
we have $\sup_{g\in\Hcal,\|g\|_{\Hcal}\le 1}\|g(\vcx)\|<\infty$.
\end{definition}
\begin{definition}\label{def:Kernel}
The function $\Kernel\in C^{2s}(\Xcal\times\Xcal,\Rbb^{n\times n})$ is an {\em operator-valued positive-definite Mercer kernel} \cite{singh2019learning} when for any $m\in\Nbb$, $\vcx,\vcy,\vcx_1,\ldots,\vcx_m\in\Xcal$ and $\vca_1,\ldots,\vca_m\in\Rbb^n$, we have $\Kernel(\vcy,\vcx)=\Kernel(\vcx,\vcy)^\tr$ and
$\sum_{1\le i,j \le m} \vca_i^\tr\Kernel(\vcx_i,\vcx_j)\vca_j \ge  0$.
\end{definition}
For any $\vcx\in\Xcal$, let $\Kernel_{\vcx}$ denote the function defined by $\Kernel(\vcx,\cdot):\Xcal\to\Rbb^n$.
This is called the {\em section} of kernel $\Kernel$ at $\vcx$ or the {\em feature map}.
\begin{theorem}[\cite{singh2019learning}]\label{thm:RKHS_rep}
With respect to any Mercer kernel $\Kernel\in C^{2s}(\Xcal\times\Xcal,\Rbb^{n\times n})$, there exists a SVRKHS of functions $g\in C^{s}(\Xcal,\Rbb^n)$, denoted by $\HK$ and endowed with inner product $\inner{\cdot}{\cdot}_{\HK}$ and norm $\|\cdot\|_{\HK}$, such that for any $(\vcx,\vcy)\in\Xcal\times \Rbb^n$ and  for any $\alpha\in\Zbbp^n$ with $|\alpha|\le s$, we have
\begin{itemize}
	\item[i)] $\partial_{\alpha}^1 \Kernel_{\vcx}\vcy\in\HK$, and
	\item[ii)] $\inner{g}{\partial_{\alpha}^1 \Kernel_{\vcx}\vcy}_{\HK}=\vcy^\tr \partial_{\alpha}g(\vcx)$, for all $g\in\HK$.
\end{itemize} 
The second feature is called the {\em reproducing property}.
\end{theorem}
We suppose that the kernel $\Kernel$ is suitably chosen such that function $h:\Xcal\to \Rbb^n$, defined as $h(\vcx):=\vcx$, belongs to $\HK$. A simple example is $\Kernel(\vcx,\vcy)$ defined as $\kernel(\vcx,\vcy)\eye$ where $\kernel$ is a polynomial kernel. 
Also, let assume $\mxT \in \Sbb^n_{++}$ is a positive-definite 
finite-dimensional transformation such that
$\mxT f\in\HK$, i.e., $\mxT$ is a positive-definite change of coordinates on $\Rbb^n$ which transforms the vector filed $f$ to an element of $\HK$.
More precisely, we know that $g\in\HK$ where $g:\Omega \to \Rbb^n$ is defined as $g(\vcx) = \mxT\big{(} f(\vcx)\big{)}$, for any $\vcx\in\Omega$. For example $\mxT$ might be a scaling of the identity matrix. Note that this is mainly a technical assumption which is used later to simplify the mathematical arguments.

Let the {\em fitting loss} or the {\em error function}, denoted by $\Lscr_{\Dcal,\Zcal}$, be the function $\Lscr_{\Dcal}:\HK\to\Rbb$ defined as
\begin{equation}
\Lscr_{\Dcal}(f) := \sum_{i=1}^\nS\|\vc{y}_i-f(\vc{x}_i)\|_{\mxW}^2,
\end{equation} 
where $\mxW\in\Sbb_{++}^n$ is an error weighting matrix. 
Additionally, we can consider a suitable kernel-based regularization due to $\mxT f\in\HK$. More precisely, let the regularization function $\Rscr:\HK\to \Rbb_{\ge 0}$ be defined as $\Rscr(\mxT f)  :=  \|\mxT f\|_{\HK}^2$.
The identification problem is now formulated as following 
\begin{equation}\label{eqn:reg_f_original}
\min_{f\in\Ccal}
\ 
\sum_{i=1}^\nS\Lscr_{\Dcal,\Zcal}(f) + \lambda\Rscr(\mxT f)
\end{equation}
where $\lambda>0$ is the regularization weight and $\Ccal$ is the set of smooth vector fields such that $\vc{x}=0$ is a stable equilibrium point and  attractive in the region $\Omega$, and also, for any $f\in\Ccal$, we have that $\mxT f\in\HK$. 
Note that \eqref{eqn:reg_f_original}
is a non-convex optimization problem with an infinite-dimensional feasible set and infinite  number of constraints.
In the followings, we show that this problem has a tractable reformulation.

Due to Theorem \ref{thm:grid_sufficient}, for imposing the stability feature given in \eqref{eqn:xtrPfx_le_eps_normx2}, it is sufficient to take a suitable $(\alpha,\beta)$-grid, like $\Zcal = \{\vc{z}_1,\ldots,\vc{z}_{\nG}\}$, and  solve optimization problem over the grid,
\begin{equation}\label{eqn:reg_f}
\begin{array}{cl}
\minOp_{f\in C^2(\Ucal,\Rbb^2),\ \mx{P}\in\Sbb_{++}^{n}}
&\sumOp_{i=1}^\nS\|\vc{y}_i-f(\vc{x}_i)\|_{\mxW}^2 + \lambda \|\mxT f\|^2_{\HK}
\\
\mathrm{s.t.}
&
\vc{z}_i^\tr\mx{P}f(\vc{z}_i)\le -\|\vc{z}_i\|^2, 
\qquad\qquad \forall i = 1,\ldots,\nG,
\\&
\frac{1}{2}(\mx{P}\mx{D}f(\zero)+\mx{D}f(\zero)^\tr \mx{P}) \preceq -\eye,
\\& 
f(\zero) = \zero,
\\& 
\mxT f\in\HK,
\\& 
\mx{P}\succeq\eye.
\end{array}
\end{equation}
The existence of such a grid is guaranteed by Theorem \ref{thm:grid_existence}.
Rewriting  optimization problem \eqref{eqn:reg_f} in terms of $g$, one has
\begin{equation}\label{eqn:reg_g}
\begin{array}{cl}
\minOp_{g\in\HK,\ \mx{P}\succeq\eye}
&
\
\sumOp_{i=1}^\nS\|\vc{y}_i-\mxT^{-1}g(\vc{x}_i)\|_{\mxW}^2 + \lambda \|g\|^2_{\HK}
\\
\mathrm{s.t.}
&
\vc{z}_i^\tr\mx{P}\mxT^{-1}g(\vc{z}_i)\le -\|\vc{z}_i\|^2, \qquad\qquad \forall i = 1,\ldots,\nG,
\\&
\frac{1}{2}(\mx{P}\mxT^{-1}\mx{D}g(\zero)+\mx{D}g(\zero)^\tr\mxT^{-1}
\mx{P}) \preceq -\eye,
\\& 
g(\zero) = \zero.
\end{array}
\end{equation}
The problem \eqref{eqn:reg_g} in a non-convex infinite-dimensional optimization and therefore, it is not tractable. However, in order to address this issue, we derive a finite dimensional problem equivalent to \eqref{eqn:reg_g}.

With respect to a given $\mxP\in\Sbb_{++}^n$, we define $\Fcal_{\mxP}$ as
\begin{equation}\label{eqn:FP}
\begin{split}
\Fcal_{\mxP} := \bigg{\{}g\in\HK
\ \Big{|}\ &
\vc{z}_i^\tr\mx{P}\mxT^{-1}g(\vc{z}_i)
\le -\|\vc{z}_i\|^2, 
\ \forall i = 1,\ldots,\nG,
\\&
\frac{1}{2}(\mx{P}\mxT^{-1}\mx{D}
g(\zero)+\mx{D}g(\zero)^\tr\mxT^{-\tr}
\mx{P}) \preceq  -\eye,
\quad g(\zero) = \zero
\bigg{\}}.
\end{split}
\end{equation}
\begin{theorem}\label{thm:FP_nonempty_closed_convex}
For any $\mxP\succeq\eye$, the set $\Fcal_{\mxP}$ is a non-empty, closed and convex subset of $\HK$.	
\end{theorem}
\begin{proof}
Since $\mxP$ and $\mxT$ are positive definite matrices, all of the eigenvalues of matrix $\mx{M}$ defined as $\mx{M}:=\frac{1}{2}(\mxP\mxT^{-1}+\mxT^{-\tr}\mxP)$ are strictly larger than zero.	
Let the function $g_\gamma:\Xcal\to\Rbb^n$ be defined as $g_\gamma(\vcx)=-2\gamma^{-1}\vcx$ where $\gamma$ is a positive real scalar smaller than smallest eigenvalue of $\mx{M}$. Since $-\frac{1}{2}\gamma g_\gamma\in\HK$, we know that $g_\gamma\in\HK$.
We have
\begin{equation*}
\vcz_i^\tr\mxP\mxT^{-1}g_\gamma(\vcz_i)=
-2\gamma^{-1}\vcz_i^\tr\mx{M}\vcz_i\le -\|\vc{z}_i\|^2,
\end{equation*}
and 
\begin{equation*}
\frac{1}{2}(\mxP\mxT^{-1}\mx{D}g_\gamma(\zero)+\mx{D}g_\gamma(\zero)^\tr\mxT^{-1}\mxP) = -2\gamma^{-1}\mx{M} \preceq -\eye.
\end{equation*}
Moreover, we know that $g_\gamma(\zero)=\zero$.	Therefore, $g_\gamma\in\Fcal_{\mxP}$ and thus, $\Fcal_{\mxP}$ is non-empty.
The convexity of $\Fcal_{\mxP}$ is due to the linear dependency of the left-hand sides of the constraints with respect to $g$. 
Now, let $(g_k)_{k=1}^\infty\in\Fcal_{\mxP}$ be a sequence converging to $g\in\HK$.
Let $\vcy$ be an arbitrary vector in $\Rbb^n$.
Then, for any $\vcx\in\Xcal$ and any $\alpha\in\Zbbp^n$, due to the reproducing property and the Cauchy-Schwarz inequality, we have
\begin{equation*}
\begin{split}
|\vcy^\tr(\partial_{\alpha}g(\vcx)-\partial_{\alpha}g_k(\vcx))| 
&=
|\inner{g-g_k}{\partial_{\alpha}^1 \Kernel_{\vcx}\vcy}_{\HK}|
\\&\le 
\|g-g_k\|_{\HK}
\|\partial_{\alpha}^1 \Kernel_{\vcx}\vcy\|_{\HK}.
\end{split}
\end{equation*}
This shows that $\lim_{k\to \infty}\partial_{\alpha}g_k(\vcx)=\partial_{\alpha}g(\vcx)$. 
Therefore, we have 
\begin{equation*}
\begin{split}
&
\lim_{k\to \infty}g_k(\zero)=g(\zero),
\qquad
\lim_{k\to \infty}\mx{D}g_k(\zero)=\mx{D}g(\zero),
\end{split}
\end{equation*}
and
\begin{equation*}
\lim_{k\to \infty}\partial_{\alpha}g_k(\vcz_i) =\partial_{\alpha}g(\vcz_i), \qquad \forall i=1,\ldots,\nG.
\end{equation*}
Since for any $k\in\Nbb$, $g_k$ satisfies the constraints and the left-hand sides of the constraints depend linearly on $g_k$ and $g$,  it follows that, the constraints are also satisfied by $g$, i.e., $g\in\Fcal_{\mxP}$. Hence, $\Fcal_{\mxP}$ is a closed subset of $\HK$.
\end{proof}
For ease of notation, 
define $p:=\nS+\nG, m := 1+p + n$, and also, set
$\vcx_0 = \zero$, and $\vcx_i = \vcz_{i-\nS}$, for $1+\nS\le i \le p$. 
\begin{theorem}\label{thm:min_ginFP_Risk_unique solution}
For any $\mxP\succeq\eye$ and $\lambda>0$, the optimization problem 
\begin{equation}\label{eqn:min_ginFP_loss}
\minOp_{g\in\Fcal_{\mxP}}\  \sum_{i=1}^\nS\|\vc{y}_i-\mxT^{-1}g(\vc{x}_i)\|_{\mxW}^2 + \lambda \|g\|^2_{\HK},
\end{equation}
has a {\em unique} solution, denoted by $g^*_{\mxP}$.
Moreover, there exist vectors $\{\vca_i\}_{i=0}^{p}$ and $\{\vcb_j\}_{j=1}^n$ such that, $g^*_{\mxP}$,
the solution of \eqref{eqn:min_ginFP_loss}, is in the  following form
\begin{equation}\label{eqn:rep_gstar_P}
g^*_{\mxP} 
= 
\sum_{i=0}^{p} \Kernel_{\vcx_i}\vca_i
+
\sum_{j=1}^{n} \partial_j^1\Kernel_{\zero}\vcb_j.
\end{equation}
\end{theorem}
\begin{proof}
Define $\Jcal:\HK\to\Rbb\cup\{+\infty\}$ as
\begin{equation*}\label{eqn:J}
\Jcal(g) 
:= \sum_{i=1}^\nS\|\vc{y}_i-\mxT^{-1}g(\vc{x}_i)\|_{\mxW}^2 \ +\ \lambda \|g\|^2_{\HK} \ + \ \Ical_{\Fcal_{\mxP}}(g),
\end{equation*}
for any $g\in\HK$. 
According to Theorem \ref{thm:FP_nonempty_closed_convex}, $\Fcal_{\mxP}$ is a non-empty, closed, and convex set, and therefore  $\Ical_{\Fcal_{\mxP}}$ is a proper lower-semicontinuous convex function \cite{peypouquet2015convex}. 
Let $g_{\gamma}$ be the element of $\HK$ introduced in the proof of Theorem \ref{thm:FP_nonempty_closed_convex}.
Since $g_{\gamma}\in \Fcal_{\mxP}$, we have $\Ical_{\Fcal_{\mxP}}(g_{\gamma})=0$.
Also, we know that $\sum_{i=1}^\nS\|\vc{y}_i+2\gamma^{-1}\mxT^{-1}\vc{x}_i\|_{\mxW}^2<\infty$.
Therefore, $\sum_{i=1}^\nS\|\vc{y}_i-\mxT^{-1}g(\vc{x}_i)\|_{\mxW}^2$ is a proper and continuous convex function with respect to $g$. Since $\lambda>0$ and $\|g_\gamma\|_{\HK}<\infty$, we have that $\Jcal$ is a proper and lower-semicontinuous strongly convex function.
Therefore, $\min_{g\in\HK}\Jcal(g)$ has a unique (finite) solution \cite{peypouquet2015convex}, which means that \eqref{eqn:min_ginFP_loss} admits a unique solution with finite cost.  
Define set $\Vcal\subseteq\HK$ as
\begin{equation*}\label{eqn:V}
\Vcal  := \bigg{\{}
\sum_{i=0}^{p} \Kernel_{\vcx_i}\vca_i
+
\sum_{j=1}^{n} \partial_j^1\Kernel_{\zero}\vcb_j 
\
\Big{|} 
\
\vca_0,\ldots,\vca_p,\vcb_1,\ldots,\vcb_n\in\Rbb^n\bigg{\}}.
\end{equation*}
This is a finite-dimensional subspace of $\HK$ and consequently, it is a closed subspace. Hence, one can decompose $g_{\mxP}^*$ as $g_{\mxP}^* = g_{\mxP}^{\parallel}+g_{\mxP}^{\perp}$ where  $g_{\mxP}^{\parallel}\in\Vcal$ and $g_{\mxP}^{\perp}\in\Vcal^{\perp}$.
Therefore,  for any $\vcy\in\Rbb^n$ and for any $i=0,\ldots,p$, from the reproducing property, we have $\vcy^\tr g_{\mxP}^{\perp}(\vcx_i) = \inner{g_{\mxP}^{\perp}}{\Kernel_{\vcx_i}\vcy}_{\HK}=0$ and subsequently, we have $\vcy^\tr g_{\mxP}^{\parallel}(\vcx_i) = \vcy^\tr g_{\mxP}^*(\vcx_i)$. Accordingly, one can conclude that $g_{\mxP}^{\parallel}(\vcx_i) = g_{\mxP}^*(\vcx_i)$, for any $i=0,\ldots,p$.
Similarly, due to the reproducing property, we have
\begin{equation*}
\begin{array}{rll}
\vcy^\tr \mx{D}g_{\mxP}^{\perp}(\zero) 
&=&
[
\vcy^\tr \partial^1_1 g_{\mxP}^{\perp}(\zero), 
\ldots, 
\vcy^\tr \partial^1_n g_{\mxP}^{\perp}(\zero) ]
\\ &= &
[
\inner{g^{\perp}_{\mxP} }{\partial^1_1\Kernel_{\zero}\vcy}_{\HK}, 
\ldots, \inner{g^{\perp}_{\mxP} }{\partial^1_n\Kernel_{\zero}\vcy}_{\HK}] 
\\&=& \zero^\tr,
\end{array}
\end{equation*}
which shows that $\mx{D}g_{\mxP}^{\perp}(\zero)$ is zero and subsequently, $\mx{D}g_{\mxP}^{\parallel}(\zero) =\mx{D}g_{\mxP}^*(\zero)$. As $g_{\mxP}^*\in\Fcal_{\mxP}$, it follows that $g_{\mxP}^{\parallel}\in\Fcal_{\mxP}$. Also, we have 
\begin{equation*}
\sum_{i=1}^\nS\|\vc{y}_i-\mxT^{-1}g^{\parallel}(\vc{x}_i)\|_{\mxW}^2 = \sum_{i=1}^\nS\|\vc{y}_i-\mxT^{-1}g(\vc{x}_i)\|_{\mxW}^2,
\end{equation*}
and 
\begin{equation*}
\|g\|_{\HK}^2 = \|g^{\parallel}\|_{\HK}^2 +\|g^{\perp}\|_{\HK}^2 \ge \|g^{\parallel}\|_{\HK}^2.
\end{equation*} 
We need to have $\|g^{\perp}\|_{\HK}^2=0$, otherwise $g^{\parallel}$ is a feasible solution with objective value strictly smaller than minimum of the objective function. This means that $g = g^{\parallel}\in\Vcal$ and $g^*_{\mxP}$ has the form given in \eqref{eqn:rep_gstar_P}.
\end{proof} 
For simplicity of notation, we define $\vca_i = \vcb_{i-p-1}$, for $m-n\le i \le m-1$, and vector $\vcab\in\Rbb^{nm}$  as
$\vcab:=[\vca_0^\tr\ldots\vca_{m-1}^\tr]^\tr$.

In the next theorem, we introduce a finite dimensional version of \eqref{eqn:reg_g}.
First, we need to introduce required notations.
Define matrix $\mxKb\in\Rbb^{mn\times mn}$ and $\mxKb_i\in\Rbb^{n\times mn}$ respectively as
$\mxKb:=[\mxKb_{i_1,i_2}]_{i_1=0,i_2=0}^{m-1,m-1}$ and 
$\mxKb_i:=[\mxKb_{i,0},\ldots,\mxKb_{i,m-1}]$, for any $i=0,\ldots,m$, 
where $\mxKb_{i_1,i_2}$ is given, for any $i_1,i_2=0,\ldots,m$, as
\begin{equation}\label{eqn:Ki1i2}
\mxKb_{i_1,i_2}  :=  
\begin{cases}
 \Kernel(\vcx_{i_2},\vcx_{i_1}),
&   
0\le i_1,i_2\le p,\\ 
 \partial^1_{i_2-p}\Kernel(\zero,\vcx_{i_1}),
&    
0\le i_1\le p<i_2< m,\\ 
 \partial^2_{i_1-p}\Kernel(\vcx_{i_2},\zero),
&    
0\le i_2\le p<i_1< m,\\
 \partial^2_{i_1-p}\partial^1_{i_2-p}\Kernel(\zero,\zero),
&    
p+1\le i_1,i_2< m.
\end{cases} 
\end{equation}
Define the function $\mxkb:\Xcal\to\Rbb^{n\times nm}$ as
\begin{equation}
\mxkb(\vcx) :=
\big{[}\Kernel_{\vcx_0}(\vcx), 
\dots, 
\Kernel_{\vcx_{p}}(\vcx),
\partial_1^1\Kernel_{\zero}(\vcx), 
\dots, 
\partial_n^1\Kernel_{\zero}(\vcx) 
\big{]}, \qquad \forall \vc{x}\in\Xcal.
\end{equation}
Based on this definition, one can easily see that $\mxKb_i= \mxkb(\vcx_i)$, for any $i=0,\ldots,m$.
Due to the reproducing property, we have the following proposition.
\begin{proposition}
	Let $\vcab\in\Rbb^{nm}$ and $g:\Xcal\to\Rbb^n$ be defined such that $g(\vcx)=\mxkb(\vcx)\vcab$, for any $\vc{x}\in\Xcal$. 
	Then, we have that $g\in\HK$ and 
	$\|g\|_{\HK}^2 =  \vcab^\tr \mxKb \vcab$.
\end{proposition}	
\begin{theorem}\label{thm:finitedimensional_reg_opt}
Consider the following optimization problem
\begin{equation}\label{eqn:reg_g_finite}
\begin{array}{cl}
\minOp_{\substack{\mxP\succeq\eye,\ \vcab\in\Rbb^{nm}\\\mxD\in\Rbb^{n\times n}}}
&
\sumOp_{i=1}^\nS\|\vc{y}_i-\mxT^{-1}\mxKb_i\vcab\|_{\mxW}^2 + \lambda \vcab^\tr \mxKb \vcab
\\
\mathrm{s.t.}
&\vcx_i^\tr\mxP\mxT^{-1}\mxKb_i\vcab\le -\|\vcx_i\|^2, 
\qquad  
\qquad  
\forall i = p-\nG+1,\ldots,p,
\\&
\mxD = 
\begin{bmatrix}
\mxKb_{m-n+1}\vcab&\ldots&\mxKb_{m-1}\vcab&\mxKb_m\vcab
\end{bmatrix},
\\&
\frac{1}{2}(\mx{P}\mxT^{-1}\mx{D}+\mx{D}^\tr\mxT^{-1}
\mx{P}) \preceq -\eye,
\\&
\mxKb_0\vcab = \zero.
\end{array}
\end{equation}
Then, for each solution of \eqref{eqn:reg_g_finite}, one can find a solution for \eqref{eqn:reg_g} with same cost value.
\end{theorem}
\begin{proof}
One can restate optimization problem \eqref{eqn:reg_g} in the following form
\begin{equation}\label{eqn:reg_g_min_P_min_g}
\minOp_{\mx{P}\succeq\eye}
\ \bigg{(}
\minOp_{g\in\Fcal_{\mxP}} \ 
\sumOp_{i=1}^\nS\|\vc{y}_i-\mxT^{-1}g(\vc{x}_i)\|_{\mxW}^2 + \lambda \|g\|^2_{\HK}\bigg{)}.
\end{equation}
From Theorem \ref{thm:min_ginFP_Risk_unique solution}, we know that the solution of the inner problem is in the form of $g(\vcx)=\mxkb(\vcx)\vcab$, for a vector $\vcab\in\Rbb^{mn}$. Utilizing the reproducing property of the kernel and according to \eqref{eqn:Ki1i2}, substituting $g(\vcx)=\mxkb(\vcx)\vcab$ in \eqref{eqn:reg_g}, results in optimization problem \eqref{eqn:reg_g_min_P_min_g}. Solving \eqref{eqn:reg_g_min_P_min_g}, we obtain $\vcab,\mxP,\mxD$, and subsequently, $g$. 
\end{proof}
Due to Theorem \ref{thm:finitedimensional_reg_opt}, in order to solve \eqref{eqn:reg_g}, it is sufficient to find the solutions of \eqref{eqn:reg_g_finite}, which is a finite-dimensional optimization. 
One can see that  \eqref{eqn:reg_g_finite} is a finite dimensional optimization problem with a convex cost and, linear and bilinear constraints, and subsequently,  it is not a convex optimization.  

Note that the kernel $\Kernel$ is characterized by a number of constants called {\em hyperparameters}. Here, this is implicitly assumed and for the sake of more transparent discussion, we have dropped this dependency in the notations. 
The hyperparameters are required to be estimated based on the data. 
This is commonly done using a cross-validation routine. 
The hyperparameter estimation  is essentially a computationally demanding procedure, especially when the kernel has a large number of hyperparameters, which can be the case when operator-valued kernels are used. 
The non-convexity of problem increases the computational complexity of the hyperparameters estimation to the point of potential intractability.
These issues will be addressed in the following section.
\subsection{Diagonal Kernels}
In order to alleviate the issue of having a large number of hyperparameters, we take the kernel $\Kernel$ as a diagonal kernel in the form of $\Kernel(\vcx,\vcy)=\kernel(\vcx,\vcy)\eye$, where $\kernel\in C^{2s}(\Xcal\times\Xcal,\Rbb)$ is a scalar valued Mercer kernel. 
Define $\vck:\Xcal\to\Rbb^{m}$ as 
\begin{equation}
\vck(\vcx):=\begin{bmatrix}
\kernel_{\vcx_0}(\vcx)&\ldots &\kernel_{\vcx_p}(\vcx)&  \partial_1^1\kernel_{\zero}(\vcx)&\ldots &\partial_n^1\kernel_{\zero}(\vcx)
\end{bmatrix}^\tr, \qquad \forall \vc{x}\in\Xcal.
\end{equation}
By defining the matrix $\mxA\in\Rbb^{n\times m}$ as
$\mxA:=[\vca_0\ldots\vca_{m-1}]$, one can see that for the {\em unique } solution of \eqref{eqn:min_ginFP_loss}, we have $g^*_{\mxP} = \mxA \vck$. 
Similar to $\mxKb$ and $\{\mxKb_i\}_{i=0}^m$, we define matrices $\mxK\in\Rbb^{m\times m}$ and $\mxK_i\in\Rbb^{m\times 1}$ respectively as
$\mxK:=[\mxK_{i_1,i_2}]_{i_1=0,i_2=0}^{m,m}$ and 
$\mxK_i:=[\mxK_{i,0},\ldots,\mxK_{i,m}]^\tr$, for any $i=1,\ldots,m$, 
where, for each $i_1,i_2=0,\ldots,m$, $\mxK_{i_1,i_2}$  is defined similarly to \eqref{eqn:Ki1i2} but based on $\kernel$. Also, let $\mxJ\in\Rbb^{m\times n}$ be the Jacobian or derivative of $\vck$ at $\vcx=\zero$, i.e.,
\begin{equation}
\mxJ:=
\begin{bmatrix}
\partial_1\vck(\zero)&\ldots &\partial_n\vck(\zero)
\end{bmatrix}.
\end{equation}
Accordingly, we have $\mx{D}g(\zero)=\mxA \mxJ$.
Due to the reproducing property, we have the following proposition.
\begin{proposition}
Let $\mxA\in\Rbb^{m\times n}$ and $g:\Xcal\to\Rbb^n$ be defined such that $g(\vcx)=\mxA\vck(\vcx)$, for any $\vc{x}\in\Xcal$. 
Then, we have that $g\in\Hk$ and 
$\|g\|_{\Hk}^2 =  \trace\big{(}\mxA \mxK \mxA^\tr \big{)}$.
\end{proposition}	

Based on the discussion above, analogous to 
\eqref{eqn:reg_g_finite}, we can introduce the following  optimization problem
\begin{equation}\label{eqn:reg_g_finite_diagonal_kernel}
\begin{array}{cl}
\minOp_{\substack{\mxP\succeq\eye\\
\ \mxA\in\Rbb^{n\times m}}}
&
\sumOp_{i=1}^\nS\|\vc{y}_i-\mxT^{-1}\mxA\mxK_i\|_{\mxW}^2 + \lambda \ \trace\big{(}\mxA \mxK \mxA^\tr \big{)}
\\
\mathrm{s.t.}
&
\vcx_i^\tr\mxP\mxT^{-1}\mxA\mxK_i\le -\|\vcx_i\|^2, \qquad  
\qquad  
\qquad  
\forall i = p-\nG+1,\ldots,p,
\\&
\frac{1}{2}(\mx{P}\mxT^{-1}\mxA\mxJ + \mxJ^\tr\mxA^\tr\mxT^{-1}
\mx{P}) \preceq -\eye,
\\&  
\mxA\mxK_0 = \zero.
\end{array}
\end{equation}
One should note that, the complexity of this optimization problem, in terms of the number of variables and constraints, is significantly lower than \eqref{eqn:reg_g_finite}, especially when the dimension of the state space, $n$, is large.
\subsection{Towards a Convex Formulation}
In the estimation problem \eqref{eqn:reg_f}, and subsequently \eqref{eqn:reg_g}, the matrices $\mx{W}$ and $\mxT$ are introduced as arbitrary positive definite matrices. One can see that the mathematical arguments (up to Theorem   \ref{thm:finitedimensional_reg_opt}) only require the fact that $\mxW$ and $\mxT$ do not depend on $g$. This provides the opportunity of choosing them such that a change of variables lead to a convex formulation. In fact, we set $\mxW:=\mxP^2$ and $\mxT:=\mxP$. Note that for any $\vcx\in\Rbb^n$, one has $\|\vcx\|_{\mxP^2}=\|\mxP\vcx\|$. Accordingly, we have
\begin{equation}
\sum_{i=1}^\nS\|\vc{y}_i-\mxT^{-1}g(\vc{x}_i)\|_{\mxW}^2 =
\sum_{i=1}^\nS\|\mxP\vc{y}_i-g(\vc{x}_i)\|^2.
\end{equation} 
Therefore, optimization problem \eqref{eqn:reg_g} can be modified to give
\begin{equation}\label{eqn:reg_g_cvx}
\begin{array}{cl}
\minOp_{g\in\HK,\ \mx{P}\succeq\eye}
&
\sumOp_{i=1}^\nS\|\mx{P}\vc{y}_i-g(\vc{x}_i)\|^2 + \lambda \|g\|^2_{\HK}
\\
\mathrm{s.t.}
&
\vc{z}_i^\tr g(\vc{z}_i)\le -\|\vc{z}_i\|^2, \qquad\qquad\forall i = 1,\ldots,\nG,
\\&
\frac{1}{2}(\mx{D}g(\zero)+\mx{D}g(\zero)^\tr) \preceq -\eye,
\\& g(\zero) = \zero.
\end{array}
\end{equation}
One can see that \eqref{eqn:reg_g_cvx} is a convex optimization problem which by  Theorem \ref{thm:min_ginFP_Risk_unique solution} has a unique solution of the form
\begin{equation}\label{eqn:rep_gstar}
g^* 
= 
\sumOp_{i=0}^{p} \Kernel_{\vcx_i}\vca_i
+
\sumOp_{j=1}^{n} \partial_j^1\Kernel_{\zero}\vcb_j.
\end{equation}
By substituting the solution \eqref{eqn:rep_gstar} into \eqref{eqn:reg_g_cvx}, we obtain a finite problem (analogous to \eqref{eqn:reg_g_finite}) as
\begin{equation}\label{eqn:reg_g_finite_cvx}
\begin{array}{cl}
\minOp_{\substack{\mxP\succeq\eye,\ \vcab\in\Rbb^{nm}\\\mxD\in\Rbb^{n\times n}}}
&
\sumOp_{i=1}^\nS\|\mxP\vc{y}_i-\mxKb_i\vcab\|^2 + \lambda \vcab^\tr \mxKb \vcab
\\
\mathrm{s.t.}
&
\vcx_i^\tr\mxKb_i\vcab\le -\|\vcx_i\|^2, \qquad\qquad\forall i = p - \nG + 1,\ldots,p,
\\&
\mxD = 
\begin{bmatrix}
\mxKb_{m - n + 1}\vcab&\ldots&\mxKb_{m-1}\vcab&\mxKb_m\vcab
\end{bmatrix},
\\&
\frac{1}{2}(\mx{D}+\mx{D}^\tr) \preceq -\eye,
\\&
\mxKb_0\vcab = \zero.
\end{array}
\end{equation}
In the case of diagonal kernels, the modified version of \eqref{eqn:reg_g_finite_diagonal_kernel} is the following
\begin{equation}\label{eqn:reg_g_finite_diagonal_kernel_cvx}   
\begin{array}{cl}
\minOp_{\substack{\mxA\in\Rbb^{n \times  m} ,\ \mxP\succeq\eye}}
&
\sum_{i=1}^\nS \|\mxP\vc{y}_i-\mxA\mxK_i\|^2 + \lambda \   \trace\big{(}\mxA \mxK \mxA^\tr \big{)}
\\
\mathrm{s.t.}
&\vcx_i^\tr\mxA\mxK_i\le -\|\vcx_i\|^2 , 
\qquad \qquad \forall i = p - \nG + 1,\ldots, p,
\\&
\frac{1}{2}(\mxA\mxJ+\mxJ^\tr\mxA^\tr) \preceq -\eye, 
\\&
\mxA\mxK_0 = \zero.
\end{array}
\end{equation}

Based on the above discussion, we have the following proposition.
\begin{proposition}
The optimization problems \eqref{eqn:reg_g_finite_cvx} and \eqref{eqn:reg_g_finite_diagonal_kernel_cvx} are feasible convex programs with strongly convex objective functions. Therefore, each of \eqref{eqn:reg_g_finite_cvx} and \eqref{eqn:reg_g_finite_diagonal_kernel_cvx} admits a unique solution.
\end{proposition}
\begin{remark}
Due to convexity of optimization problems \eqref{eqn:reg_g_finite_cvx} and \eqref{eqn:reg_g_finite_diagonal_kernel_cvx}, one can use off-the-shelf optimization solvers, like CVX \cite{grant2014cvx}, in order to obtain the solutions.
\end{remark}
\begin{remark}
One should note that once $\mxP$ and $g$ are obtained, the vector field $f$ is calculated as $\mxP^{-1}g$.
\end{remark}

\section{Numerical Experiments}
In this section, we discuss a numerical example.
To this end, consider the dynamical system
defined as 
\begin{equation}\label{eqn:xdot=f(x)_example}
\begin{array}{rcl}
\dot{x}_1 &=& -5x_2-4x_1  +x_1x_2^2-6x_1^3, \\
\dot{x}_2 &=& -20x_1-4x_2 + 4 x_1^2x_2 + x_2^3.  \\
\end{array}
\end{equation}
For system \eqref{eqn:xdot=f(x)_example}, $(x_1,x_2)=\zero$ is a stable equilibrium point, which is attracting in region $\Omega$ defined as $\Omega:=\Bcal(\zero,1.5)$.
Assume that this information is provided as prior knowledge. 
Moreover, we consider two trajectories of the system, starting from  $\vcx_0=(1,-1)$ and  $\vcx_0=(-1,-1))$, and take samples from each of them at different $19$ locations with an additive measurement noise of $\Ncal(0,\sigma^2)$ where $\sigma^2 = 0.001$. 
For identifying \eqref{eqn:xdot=f(x)_example}, we take two approaches:~ 1) we utilize the prior knowledge on the ROA and solve  \eqref{eqn:reg_g_finite_diagonal_kernel_cvx},~
2) only the stability of the equilibrium point is considered and we solve a modified version of \eqref{eqn:reg_g_finite_diagonal_kernel_cvx} where $\mxP=\eye$ and the grid constraints are removed.
Let the corresponding solutions be denoted by $\hat{f}$ and $\check{f}$, respectively. 
The main difference of these two approaches is the inclusion of the prior knowledge on the region of attraction in the estimation method. Accordingly, the comparison of $\hat{f}$ and $\check{f}$ can reflect the impact and the potential leverage of using the prior knowledge of the ROA on the estimation. 
For $\Zcal$, we take a uniform polar grid of size $\nG=300$ inside $\Omega$.
Given these settings, we obtain estimations $\hat{f}$ and $\check{f}$. The results are shown in Figure \ref{fig:f}. The calculated {coefficient of determination} in the unit square, also known as {\em R squared}, for $\hat{f}$ and $\check{f}$ is $93.6\%$ and  $80.4\%$, respectively.

\begin{figure}[t]
	\centering	
	\begin{subfigure}{.8\textwidth}
		\centering
		\includegraphics[width=0.9\linewidth]{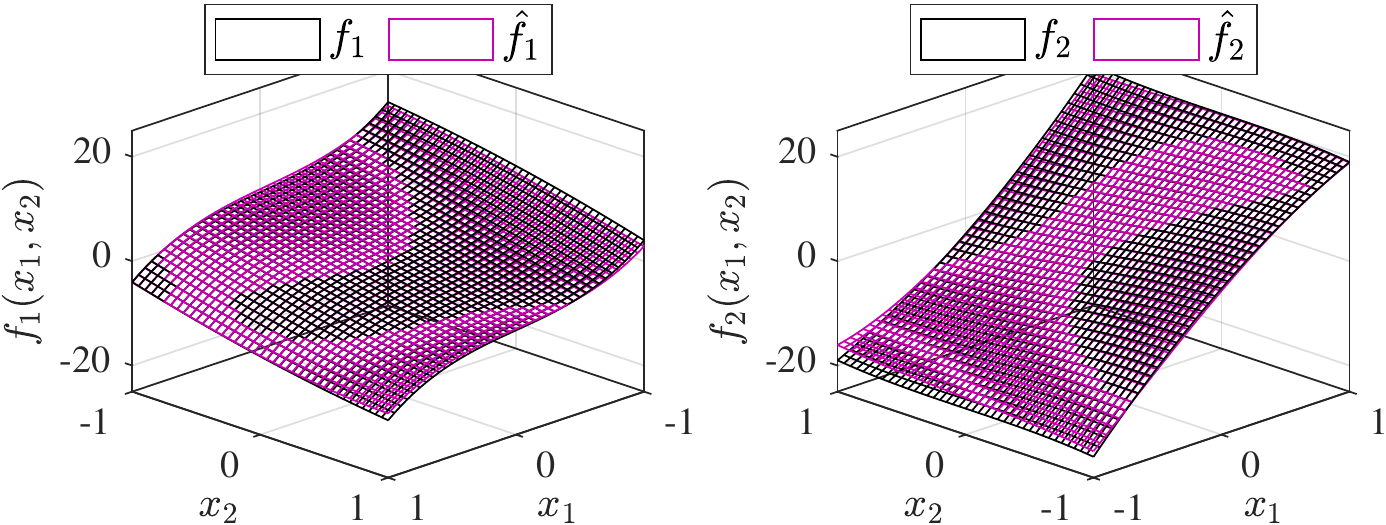} 
		\label{fig:sub-first}
	\end{subfigure}\\	\vspace{5mm}	
	\centering	
	\begin{subfigure}{.8\textwidth}
		\centering
		\includegraphics[width=.9\linewidth]{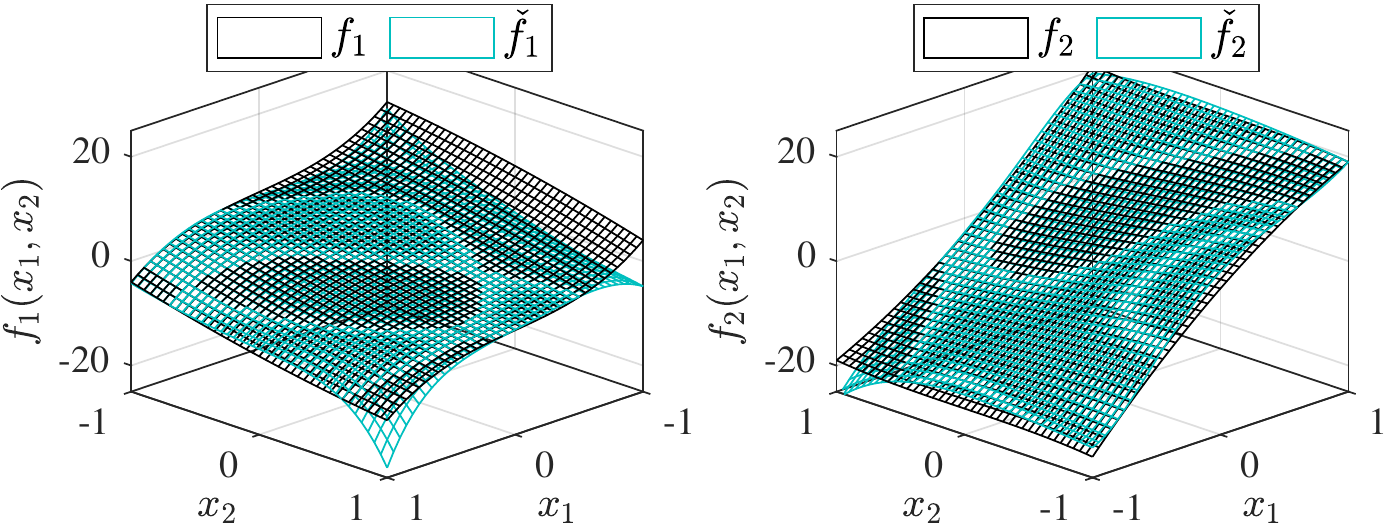} 
		\label{fig:sub-second}
	\end{subfigure}
	\caption{The figure shows coordinates of vector fields $f$ (black),  $\hat{f}$ (magenta), the estimation with ROA prior knowledge, and  $\check{f}$ (green), estimation without the ROA prior knowledge.}
	\label{fig:f}
\end{figure}

Let $\vcx_f$, $\vcx_{\hat{f}}$ and $\vcx_{\check{f}}$ denote the trajectories generated from $f$, $\hat{f}$ and $\check{f}$, respectively.
We consider initial points $\vcx_0=(-1,1),(-1.45,0),(1,1.5)$. The corresponding trajectories are shown in Figure \ref{fig:traj}. For point $\vcx_0=(-1,1)\in\Omega$, all of the trajectories goes to the equilibrium point $(0,0)$. One can see that trajectory $\vcx_{\hat{f}}$ stays close to the trajectory of true system, $\vcx_f$, while $\vcx_{\check{f}}$ deviates from $\vcx_f$ significantly.
Starting from $\vcx_0=(-1.45,0)\in\Omega$, trajectories  $\vcx_f$ and $\vcx_{\check{f}}$ stay close to each other and converge to $(0,0)$, meanwhile $\vcx_{\check{f}}$ diverges. This confirms that the prior knowledge is satisfied by the estimated vector field $\hat{f}$.
Finally, if $\vcx_0=(-1.45,0)\notin\mathrm{ROA}$, trajectory $\vcx_f$ as well as trajectory $\vcx_{\hat{f}}$ diverge. However, $\vcx_{\check{f}}$ converges to $(0,0)$ which is not expected. 

\begin{figure}[t]
	\centering
	\includegraphics[width=0.45\textwidth]{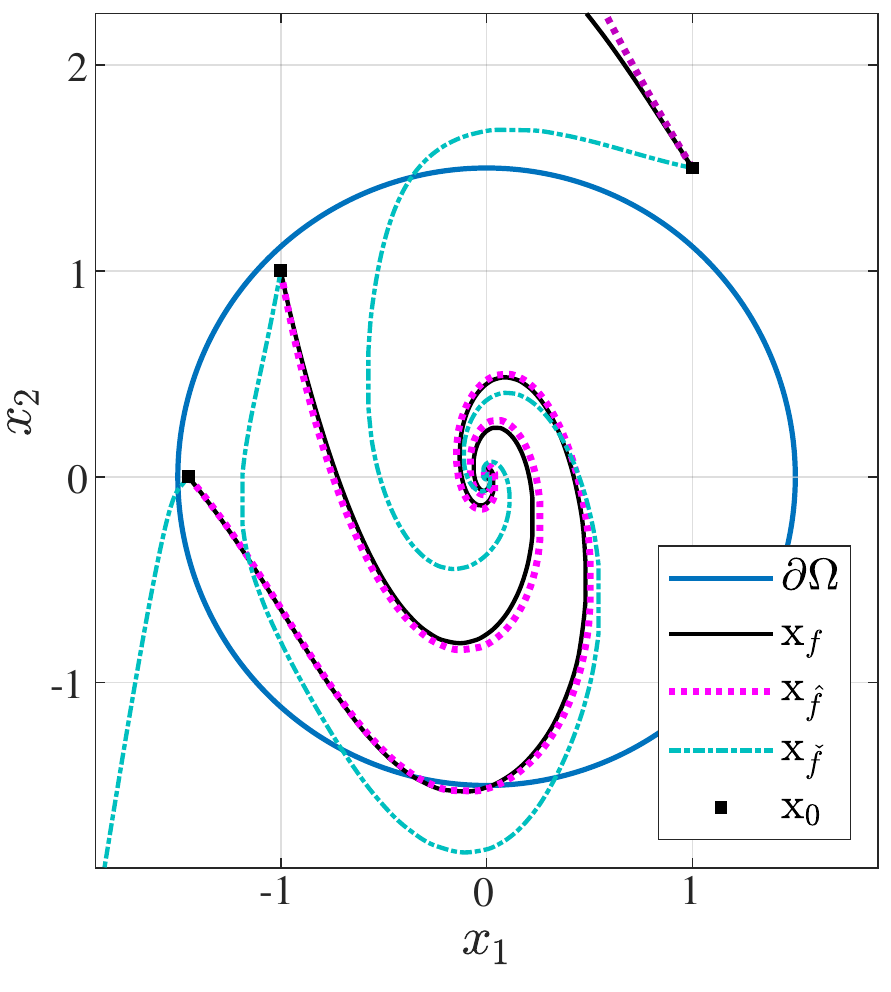}
	\caption{The different trajectories generated by the system (solid black), the estimated system with ROA prior knowledge, $\hat{f}$, (dotted magenta) and the estimation without the prior knowledge,  $\check{f}$, (dashed cyan). The circle shows the boundary of $\Omega$ which is the prior knowledge of the ROA.}
	\label{fig:traj}
\end{figure}
	
\section{Conclusion}\label{sec:conc}
We have discussed nonlinear system identification  when in addition to the measurement data, prior knowledge is available on a subset of the region of attraction (ROA) of an equilibrium point. The proposed identification method is an optimization problem minimizing the fitting error and guaranteeing the desired stability property. The resulting problem is a joint identification of the dynamics as well as a Lyapunov function for the stability property. 
Due to the functional hypothesis space for the dynamics and the Lie derivative inequalities for the stability, a bilinear infinite-dimensional optimization problem with infinite number of constraints is obtained. 
To get a tractable formulation, we consider a sufficient suitable finite subset of the constraints. The resulting problem admits a solution in form of a linear combination of the sections of the kernel and its derivatives. Subsequently, by a change of variable, we obtain a convex reformulation of the problem. Moreover, for reducing the number of hyperparameters, the optimization problem is adapted to the case of diagonal kernels. 
We have verified the approach and illustrated the results on an example. 
The fitting error for the estimation is low, and the estimated vector field confirms the expected behavior in the given subset of the ROA. 
In order to assess the impact of the prior knowledge, we have compared the method with an approach where the prior knowledge is not exploited. The comparison confirms the significance of the impact of the prior knowledge on the precision of the estimation and on the global behavior of the estimated system.
\bibliographystyle{IEEEtran}
\bibliography{mybib}
\end{document}